\documentclass[11pt]{article}

\usepackage[latin1]{inputenc} 
\usepackage{amsthm,amsmath,amssymb} 
\usepackage{graphicx}
\usepackage{color}
\usepackage{verbatim}
\usepackage{caption}

\usepackage{enumitem} 
\setlist[enumerate]{topsep=0pt,itemsep=-1ex,partopsep=1ex,parsep=1ex}

\theoremstyle{plain}
\newtheorem{theo}{Theorem}[section]

\newtheorem{lemma}[theo]{Lemma}

\theoremstyle{definition}
\newtheorem{defn}[theo]{Definition}

\newcommand{\mc}[1]{\mathcal{#1}}
\newcommand{\mb}[1]{\mathbb{#1}}
\newcommand{\nib}[1]{\noindent {\bf #1}}

\newcommand{\bfl}[1]{\left\lfloor #1 \right\rfloor}

\newcommand{\sub}{\subseteq}

\newcommand{\sm}{\setminus}
\newcommand{\ov}{\overline}

\newcommand{\eps}{\varepsilon}
\newcommand{\es}{\emptyset}

\newcommand{\gG}{\gamma}
\newcommand{\dD}{\delta}

\newcommand{\GG}{\Gamma}
\newcommand{\DD}{\Delta}

\newcommand{\bE}{\mathbb{E}}
\newcommand{\bN}{\mathbb{N}}
\newcommand{\bP}{\mathbb{P}}

\newcommand{\cA}{\mathcal{A}}
\newcommand{\cB}{\mathcal{B}}
\newcommand{\cE}{\mathcal{E}}
\newcommand{\cF}{\mathcal{F}}
\newcommand{\cG}{\mathcal{G}}
\newcommand{\cI}{\mathcal{I}}
\newcommand{\cX}{\mathcal{X}}

\newcommand{\supp}{\text{supp}}

\title{Rainbow factors in hypergraphs}

\author{
Matthew Coulson\thanks{School of Mathematics, University of Birmingham, Birmingham, UK. Email: mjc685@bham.ac.uk.} \and
Peter Keevash\thanks{Mathematical Institute,
University of Oxford, Oxford, UK. Email: keevash@maths.ox.ac.uk.
\newline \hspace*{1.8em}Research supported
in part by ERC Consolidator Grant 647678.} \and
Guillem Perarnau\thanks{School of Mathematics, University of Birmingham, Birmingham, UK. Email: g.perarnau@bham.ac.uk.} \and
Liana Yepremyan\thanks{Mathematical Institute,
University of Oxford, Oxford, UK. Email: liana.yepremyan@maths.ox.ac.uk.
\newline \hspace*{1.8em}Research supported
in part by ERC Consolidator Grant 647678}
}

\begin{document}

\maketitle

\begin{abstract}
For any $r$-graph $H$, we consider the problem of 
finding a rainbow $H$-factor in an $r$-graph $G$ 
with large minimum $\ell$-degree and an edge-colouring
that is suitably bounded. We show that the asymptotic degree threshold
is the same as that for finding an $H$-factor.
\end{abstract}

\section{Introduction}

A fundamental question in Extremal Combinatorics is to determine conditions 
on a hypergraph $G$ that guarantee an embedded copy of some other hypergraph $H$.
The Tur\'an problem for an $r$-graph $H$ asks for the maximum number of edges 
in an $r$-graph $G$ on $n$ vertices; we usually think of $H$ as fixed and $n$ as large. 
For $r=2$ (ordinary graphs) this problem
is fairly well understood (except when $H$ is bipartite), but for general $r$
and general $H$ we do not even have an asymptotic understanding 
of the Tur\'an problem (see the survey \cite{Kturan}).
For example, a classical theorem of Mantel determines the maximum number of edges
in a triangle-free graph on $n$ vertices (it is $\bfl{n^2/4}$),
but we do not know even asymptotically the maximum number of edges
in a tetrahedron-free $3$-graph on $n$ vertices. 
On the other hand, if we seek to embed a spanning hypergraph 
then it is most natural to consider minimum degree conditions.
Such questions are known as Dirac-type problems, after the classical theorem
of Dirac that any graph on $n \ge 3$ vertices with minimum degree at least $n/2$
contains a Hamilton cycle. There is a large literature on such problems for graphs
and hypergraphs, surveyed in \cite{KOicm, KOlarge,RR,yi}.

One of these problems, finding hypergraph factors, 
will be our topic for the remainder of this paper.
To describe it we introduce some notation and terminology.
Let $G$ be an $r$-graph on $[n]=\{1,\dots,n\}$.
For any $L \sub V(G)$ the degree of $L$ in $G$
is the number of edges of $G$ containing $L$.
The minimum $\ell$-degree $\dD_\ell(G)$ is
the minimum degree in $G$ over all $L \sub V(G)$ of size $\ell$.
Let $H$ be an $r$-graph with $|V(H)|=h \mid n$. 
A partial $H$-factor $F$ in $G$ of size $m$ 
is a set of $m$ vertex-disjoint copies of $H$ in $G$.
If $m=n/h$ we call $F$ an $H$-factor.
We let $\dD_\ell(H,n)$ be the minimum $\dD$ such that
$\dD_\ell(G) > \dD n^{r-\ell}$ guarantees an $H$-factor in $G$.
Then the asymptotic $\ell$-degree threshold for $H$-factors is 
\[ \dD^*_\ell(H):= \liminf_{m\to \infty} \dD_\ell(H, mh)\;. \]
 We refer to Section 2.1 in~\cite{yi} for a summary of the known bounds on $\dD^*_\ell(H)$.
As for the Tur\'an problem, 
$\dD^*_\ell(H)$ is well-understood for graphs~\cite{KSSalonyus,KOfactors},
but there are few results for hypergraphs.
Even for perfect matchings (the case when $H$ is a single edge)
there are many cases for which
the problem remains open (this is closely connected to the
Erd\H{o}s Matching Conjecture \cite{ematch}).

Let us now introduce colours on the edges of $G$ and ask for 
conditions under which we can embed a copy of $H$ that is
{\em rainbow}, meaning that its edges have distinct colours.
Besides being a natural problem in its own right,
this general framework also encodes many other 
combinatorial problems. Perhaps the most well-known of these
is the Ryser-Brualdi-Stein Conjecture \cite{brualdi,ryser,stein}
on transversals in latin squares, which is equivalent 
to saying that any proper edge-colouring of the complete bipartite 
graph $K_{n,n}$ has a rainbow matching of size $n-1$.
There are several other well-known open problems that can be
encoded as finding certain rainbow subgraphs in graphs with an
edge-colouring that is locally $k$-bounded for some $k$,
meaning that each vertex is in at most $k$ edges of any given colour
(so $k=1$ is proper colouring). For example, a recent result 
of Montgomery, Pokrovskiy and Sudakov \cite{MPS}
shows that any locally $k$-bounded edge-colouring of $K_n$
contains a rainbow copy of any tree of size at most $n/k - o(n)$,
and this implies asymptotic solutions to the conjectures 
of Ringel \cite{R} on decompositions by trees and of
Graham and Sloane \cite{GS} on harmonious labellings of trees.

We now consider rainbow versions of the extremal problems discussed above. 
The rainbow Tur\'an problem for an $r$-graph $H$ is to determine the maximum 
number of edges in a properly edge-coloured $r$-graph $G$ on $n$ vertices
with no rainbow $H$. This problem was introduced by
Keevash, Mubayi, Sudakov and Verstra\"ete \cite{KMSV},
who were mainly concerned with degenerate Tur\'an problems
(the case of even cycles encodes a problem from Number Theory),
but also observed that a simple supersaturation argument
shows that the threshold for non-degenerate rainbow Tur\'an problems
is asymptotically the same as that for the usual Tur\'an problem,
even if we consider locally $o(n)$-bounded edge-colourings.

For Dirac-type problems, it seems reasonable to make stronger assumptions
on our colourings, as we have already noted that even locally bounded
colourings of complete graphs encode many problems that are still open.
For example, Erd\H{o}s and Spencer \cite{ES} showed the existence
of a rainbow perfect matching in any edge-colouring of $K_{n,n}$
that is $(n-1)/16$-bounded, meaning that are at most $(n-1)/16$
edges of any given colour. Coulson and Perarnau \cite{CG} recently
obtained a Dirac-type version of this result, showing that
any $o(n)$-bounded edge-colouring of a subgraph of $K_{n,n}$
with minimum degree at least $n/2$ has a rainbow perfect matching.
One could consider this a `local resilience' version
(as in \cite{SV}) of the Erd\H{o}s-Spencer theorem.
This is suggestive of a more general phenomenon, 
namely that for any Dirac-type problem, the rainbow problem 
for bounded colourings should have asymptotically the same degree threshold 
as the problem with no colours. A result of Yuster~\cite{yuster}
on $H$-factors in graphs adds further evidence (but only for the weaker
property of finding an $H$-factor in which each copy of $H$ is rainbow).
For graph problems, the general phenomenon was
recently confirmed in considerable generality by Glock and Joos \cite{GJ},
who proved a rainbow version of the blow-up lemma of
Koml\'os, S\'ark\"ozy and Szemer\'edi \cite{KSS}
and the Bandwidth Theorem of B\"ottcher, Schacht and Taraz \cite{BST}.

Our main result establishes the same phenomenon for hypergraph factors. 
We will use the following boundedness assumption for our colourings,
in which we include the natural $r$-graph generalisations of
both the local boundedness and boundedness assumptions from above
(for $r=2$ boundedness implies local boundedness,
but in general they are incomparable assumptions).

\begin{defn} \label{def:bdd}
An edge-colouring of an $r$-graph on $n$ vertices
is $\mu$-\emph{bounded} if for every colour $c$:
\begin{enumerate}
\item  there are at most $\mu n^{r-1}$ edges of colour $c$,
\item  for any set $I$ of $r-1$ vertices, 
there are at most $\mu n$ edges of colour $c$ containing $I$.
\end{enumerate}
\end{defn}

Note that we cannot expect any result without some ``global'' condition as in Definition \ref{def:bdd}.i, since any $H$-factor contains linearly many edges. Similarly, some ``local'' condition along the lines of Definition \ref{def:bdd}.ii is also needed.
Indeed, consider the edge-colouring of the complete $r$-graph $K^r_n$
by $\tbinom{n}{r-1}$ colours identified with $(r-1)$-subsets of $[n]$,
where each edge is coloured by its $r-1$ smallest elements.
Suppose $H$ has the property that every $(r-1)$-subset of $V(H)$
is contained in at least $2$ edges of $H$ (e.g.\ suppose $H$ is also complete).
Then there are fewer than $n$ edges of any given colour, but
there is no rainbow copy of $H$ (let alone an $H$-factor), as in any 
embedding of $H$ all edges containing the $r-1$ smallest elements have the same colour.

Our main theorem is as follows (we use the notation $a \ll b$ to mean that
for any $b>0$ there is some $a_0>0$ such that the statement holds for $0<a<a_0$).

\begin{theo}
\label{mainthm}
Let $1/n \ll \mu \ll \eps \ll 1/h \leq 1/r <1/\ell \leq 1$ with $h|n$. 
Let $H$ be an $r$-graph on $h$ vertices 
and $G$ be an $r$-graph on $n$ vertices 
with $\dD_\ell(G) \geq (\dD^*_\ell(H)+\eps)n^{r-\ell}$. 
Then any $\mu$-bounded edge-colouring of $G$ admits a rainbow $H$-factor.
\end{theo}

Throughout the remainder of the paper we fix $\ell$, $r$, $h$, $\eps$, $\mu$, 
$n$, $H$ and $G$ as in the statement of Theorem \ref{mainthm}.
We also fix an integer $m$ with $\mu \ll 1/m \ll \eps$
and define $\gG = (mh)^{-m}$.

\section{Proof modulo lemmas}

The outline of the proof of Theorem \ref{mainthm} 
is the same as that given by Erd\H{o}s and Spencer \cite{ES}
for the existence of Latin transversals:
we consider a uniformly random $H$-factor $\mc{H}$ in $G$
(there is at least one by definition of $\dD^*_\ell(H)$)
and apply the Lopsided Lov\'asz Local Lemma (Lemma \ref{L4})
to show that $\mc{H}$ is rainbow with positive probability.
We will show that the local lemma hypotheses hold
if there are many feasible switchings, defined as follows.

\begin{defn} \label{def:feas}
Let $F_0$ be an $H$-factor in $G$ and $H_0 \in F_0$.
An $(H_0,F_0)$-switching is a partial $H$-factor 
$Y$ in $G$ with $V(H_0) \sub V(Y)$ such that
\begin{enumerate}
\item for each $H' \in F_0$
we have $V(H') \sub V(Y)$ or $V(H') \cap V(Y) = \es$, and 
\item each $H' \in Y$ shares at most one vertex with $H_0$. 
\end{enumerate}
Let $Y'$ be obtained from $Y$ by deleting 
all vertices in $V(H_0)$ and their incident edges.
We call $Y$ feasible if $Y'$ is rainbow and does not 
share any colour with any $H' \in F_0 \sm V(Y)$.
\end{defn}

The following lemma, proved in Section \ref{sec:key},
reduces the proof of Theorem \ref{mainthm}
to showing the existence of many feasible switchings.

\begin{lemma} \label{keylemma}
Suppose that for every 
$H$-factor $F_0$ of $G$ and $H_0 \in F_0$ there are at least 
$\gG n^{m-1}$ feasible $(H_0,F_0)$-switchings of size $m$.
Then $G$ has a rainbow $H$-factor.
\end{lemma}

We will construct switchings by randomly choosing 
some copies of $H$ from $F_0$ and considering 
a random transverse partition in the sense 
of the following definition.

\begin{defn} \label{def:suit}
Let $F_0$ be an $H$-factor in $G$ and $H_0 \in F_0$. Let $X \sub F_0$ be a partial $H$-factor in $G$ with $H_0 \in X$. We call $S \sub V(X)$ transverse
if $|H' \cap S| \le 1$ for all $H' \in X$.
We call a partition of $V(X)$ transverse if each part is transverse. 
For any edges $e$ and $f$ let 
$X(e,f) = \{H' \in X: |V(H') \cap (e \cup f)| \ge 2\}$.
We call $X$ suitable if
\begin{enumerate}
\item for any transverse $I \sub V(X) \sm V(H_0)$ with $|I|=r-1$
there are at most $\eps |X|/4$ vertices $v \in V(X)$ such that 
$I \cup \{v\} \in E(G)$ shares a colour with some $H' \in F_0$, and
\item for any transverse edges $e$ and $f$ disjoint from $V(H_0)$
of the same colour we have $X(e,f) \ne \es$, and furthermore
if $e \cap f = \es$ then $|X(e,f)| \ge 2$.
\end{enumerate}
\end{defn}

The following lemma, proved in Section \ref{sec:feas},
shows that a suitable partial $H$-factor has
an associated feasible switching if it has 
a transverse partition whose parts each satisfy
the minimum degree condition for an $H$-factor.

\begin{lemma} \label{feas}
Let $F_0$, $H_0$ and $X$ be as in Definition \ref{def:suit}, 
suppose $X$ is suitable and $|X|=m$.
Let $\mc{P} = (V_1,\dots,V_h)$ be a transverse partition of $V(X)$
and suppose $\dD_\ell(G[V_i]) \geq (\dD^*_\ell(H) + \eps/2) m^{r-\ell}$
for all $i \in [h]$.
Then there is a partial $H$-factor $Y$ in $G$ with $V(Y)=V(X)$
such that $Y$ is a feasible $(H_0,F_0)$-switching.
\end{lemma}

The following lemma, proved in Section \ref{sec:shuffle},
gives a lower bound on the number of partial $H$-factors $X$ with
some transverse partition $\mc{P}$ satisfying the conditions of the previous lemma.

\begin{lemma} \label{shuffle}
Let $F_0$ be an $H$-factor in $G$ and $H_0 \in F_0$.
Let $X \sub F_0$ be a random partial $H$-factor 
where $H_0 \in X$ and each $H' \in F_0 \sm \{H_0\}$
is included independently with probability $p=\frac{m}{n/h-1}$.
Let $\mc{P} = (V_1,\dots,V_h)$ be a uniformly random
transverse partition of $V(X)$.
Then with probability at least $1/m$ 
we have $X$ suitable, $|X|=m$ and all
$\dD_\ell(G[V_i]) \geq (\dD^*_\ell(H) + \eps/2) m^{r-\ell}$.
\end{lemma}

We conclude this section by showing how Theorem \ref{mainthm}
follows from the above lemmas.
\begin{proof}[Proof of Theorem~\ref{mainthm}]
By Lemma \ref{keylemma}, it suffices to show that for every 
$H$-factor $F_0$ of $G$ and $H_0 \in F_0$ there are at least 
$\gG n^{m-1}$ feasible $(H_0,F_0)$-switchings of size $m$.
There are $\binom{n/h-1}{m-1}\geq (n/mh - 1)^{m-1}$ 
partial $H$-factors $X$ of size $m$
with $H_0 \in X \sub F_0$.
By Lemma \ref{shuffle}, at least 
$m^{-1} (n/mh - 1)^{m-1} > \gG n^{m-1}$ of these
are suitable and have a transverse partition
$\mc{P} = (V_1,\dots,V_h)$ with all
$\dD_\ell(G[V_i]) \geq (\dD^*_\ell(H) + \eps/2) m^{r-\ell}$.
By Lemma \ref{feas}, each such $X$ has
an associated feasible $(H_0,F_0)$-switching.
\end{proof}

\section{Probabilistic methods}

In this section we collect various probabilistic tools that
will be used in the proofs of the lemmas stated in the previous section.
We start with a general version of the local lemma 
which follows easily from that given by Spencer \cite{S}.

\begin{defn} \label{def:dep}
Let $\cE$ be a set of events in a finite probability space.
Suppose $\GG$ is a graph with $V(\GG) = \mc{E}$ and $p \in [0,1]^{\mc{E}}$.
We call $\GG$ a \emph{$p$-dependency graph for $\cE$}
if for every $E \in \mc{E}$ and $\mc{E}' \sub \mc{E}$
such that $EE' \notin E(\GG)$ for all $E' \in \mc{E}'$ and $\bP[\cap_{E' \in \mc{E}'} \ov{E'}]>0$,
we have $\bP[E \vert \cap_{E' \in \mc{E}'} \ov{E'}] \leq p_E$.
\end{defn}

\begin{lemma} \label{L4}
Under the setting of Definition \ref{def:dep},
if $\sum \{ p_{E'} : EE' \in E(\GG) \} \le 1/4$
for all $E \in \cE$ then with positive probability
none of the events in $\cE$ occur.
\end{lemma}

We also require Talagrand's Inequality, see e.g.\ \cite[page 81]{MR}.

\begin{lemma} \label{talagrand2}
Let $X \ge 0$ be a random variable determined by $n$ independent trials, such that:
\begin{description}
\item[\ $c$-Lipschitz.] Changing the outcome of any one trial can affect $X$ by at most $c$.
\item[\ $r$-certifiable.] If $X \geq s$ then there is a set 
of at most $rs$ trials whose outcomes certify $X \geq s$.
\end{description}
Then for any $0 \leq t \leq \bE[X]$,
\begin{equation*}
\bP[|X-\bE[X]| > t + 60 c \sqrt{r \bE[X]}] \leq 4 e^{-t^2/(8c^2 r \bE[X])} . 
\end{equation*}
\end{lemma}

Next we state an inequality of Janson \cite{J}.

\begin{defn} \label{def:strongdep}
Let $\{I_i\}_{i\in \cI}$ be a finite family of indicator random variables.
We call a graph $\GG$ on $\cI$ a strong dependency graph if 
the families $\{I_i\}_{i\in A}$ and  $\{I_i\}_{i\in B}$ are independent
whenever $A$ and $B$ are disjoint subsets of $\cI$
with no edge of $\GG$ between $A$ and $B$.
\end{defn}

\begin{theo} \label{janson} 
In the setting of Definition \ref{def:strongdep},
let $S=\sum_{i\in \cI}{I_i}$, $\mu=\bE[S]$, 
$\dD=\max_{i\in \cI}\sum \{ p_j : ij \in E(\GG) \}$ 
and $\DD = \sum \{ \bE[I_iI_j] : ij \in E(\GG) \}$. 
Then for any $0<{\eta} <1$, \[\bP[S<(1-{\eta}) \mu ]\leq 
\exp(-\min\{ ({\eta}\mu)^2/(8\Delta +2\mu), {\eta}\mu/(6\dD) \}).\]
\end{theo}
  
We conclude with a standard bound on the
probability that a binomial is equal to its mean.

\begin{lemma}
\label{binbound}
Let $X$ be a binomial random variable 
with parameters $n$ and $p$ 
such that $np = m \in \bN$ and $m^2 = o(n)$.
Then $\bP[X = m] \geq 1/(4\sqrt{m})$.
\end{lemma}
\begin{proof}
The stated bound follows from
$\bP[X = m] = \tbinom{n}{m} p^m (1-p)^{n-m} 
\ge m!^{-1} (n-m)^m p^m (1-p)^{n-m} = m!^{-1} m^m (1-p)^n$,
$(1-p)^n = e^{-np + O(np^2)}$ and $m! \leq e^{1-m} m^{m+1/2}$.
\end{proof}

\section{Applying the local lemma} \label{sec:key}

In this section we prove Lemma \ref{keylemma},
which applies the local lemma to reduce the proof of 
Theorem \ref{mainthm} to finding many feasible switchings.

\begin{proof}[Proof of Lemma \ref{keylemma}]
Suppose that for every 
$H$-factor $F_0$ of $G$ and $H_0 \in F_0$ there are at least 
$\gG n^{m-1}$ feasible $(H_0,F_0)$-switchings of size $m$.
We need to show that $G$ has a rainbow $H$-factor.

We will apply Lemma \ref{L4} to
a uniformly random $H$-factor $\mc{H}$ in $G$, where 
$\mc{E} = \mc{A} \cup \mc{B}$ 
consists of all events of the following two types.
For every copy $H_0$ of $H$ in $G$ and any two edges 
$e$ and $f$ in $H_0$ of the same colour 
we let $A(e,f,H_0)$ be the event that $H_0 \in \mc{H}$;
we let $\mc{A}=\{A(e,f,H_0): \bP[A(e,f,H_0)]>0\}$.
For every pair $H_1, H_2$ of vertex-disjoint copies of $H$ in $G$
and edges $e_1$ of $H_1$ and $e_2$ of $H_2$ of the same colour 
we let $B(e_1,e_2,H_1,H_2)$ be the event that
$H_1 \in \mc{H}$ and $H_2 \in \mc{H}$;
we let $\cB={\{B(e_1,e_2,H_1,H_2): \bP[B(e_1,e_2,H_1,H_2)]>0\}}$. 
Then $\mc{H}$ is rainbow iff none of the events in $\cE$ occur.

We define the supports of $A=A(e,f,H_0)$ as $\supp(A)=V(H_0)$
and of $B={B(e_1,e_2},H_1,H_2)$ as ${\supp(B)=}V(H_1)\cup V(H_2)$. 
Let $\GG$ be the graph on $\cA\cup\cB$ where $E_1,E_2\in V(\GG)$ 
are adjacent if and only if $\supp(E_1)\cap\supp(E_2)\neq\es$.
Our goal is to show that there exist suitably small $p_\cA,p_\cB$ 
such that $\GG$ is a $p$-dependency graph for $\cA\cup \cB$, 
where $p_A=p_\cA$ for all $A\in \cA$ and $p_B=p_\cB$ for all $B\in \cB$. 
For $\cX \in \{\cA,\cB\}$, let $d_{\cX}$ be the maximum over $E \in V(\GG)$
of the number of neighbours of $E$ in $\cX$.
To apply Lemma \ref{L4}, it suffices to show 
$p_{\cA}d_{\cA}+p_{\cB}d_{\cB} \leq 1/4$.

To bound the degrees, we will first estimate the number of events in $\mc{A}$ 
and $\mc{B}$ whose support contains any fixed vertex $v\in V(G)$. 
We claim that there are at most $2^{{r+1}} h!  \mu n^{h-1}$
events $A(e,f,H_0){\in\cA}$ with $v \in V(H_0)$. 
To see this, first consider those events with $v \notin e \cup f$.
For any $s<r$, as the colouring is $\mu$-bounded, 
the number of choices of $e$ and $f$ of the same colour 
with $|e \cap f|=s$ is at most $n^r \cdot \tbinom{r}{s} \mu n^{r-s}$.
For any such $e$ and $f$ with $v \notin e \cup f$,
there are at most $h! n^{h-(2r-s+1)}$ copies of $H$ 
containing $e \cup f \cup \{v\}$, so summing over $s$
we obtain at most $2^r h! \mu n^{h-1}$ such events.
Now we consider events $A(e,f,H_0)$ with $v \in e \cup f$.
The number of choices of $e$ and $f$ of the same colour with $|e \cap f|=s$ 
and $v \in e \cup f$ is at most $n^{r-1} \cdot \tbinom{r}{s} \mu n^{r-s}$.
For any such $e$ and $f$ there are at most $h! n^{h-(2r-s)}$ copies of $H$ 
containing $e \cup f \cup \{v\}$, so summing over $s$
we obtain at most $2^{r+1} h! \mu n^{h-1}$ such events. The claim follows.

Similarly, we claim that there are at most {$2 (h!)^2 \mu n^{2h-2}$}
events $B(e_1,e_2,H_1,H_2){\in \cB}$ with $v \in V(H_1) \cup V(H_2)$. To see this, first consider those events 
with $v \in e_1 \cup e_2$. By definition of $\mathcal{B}$, we may consider only disjoint edges $e_1$, $e_2$.
There are at most $h!n^{h-r}$ choices for each of $H_1$ and $H_2$
given $e_1$ and $e_2$. Also, the number of choices for $e_1$ and $e_2$
is at most {$n^{r-1} \cdot \mu n^{r-1}= \mu n^{2r-2}$}.
Thus, we obtain at most $ (h!)^2 \mu n^{2h-2}$ such events.
A similar argument applies to events $B(e_1,e_2,H_1,H_2)$ 
with $v \notin e_1 \cup e_2$, and the claim follows.

In particular, there is some constant $C=C(r,h)$ so that 
\begin{align}\label{eq:d}
d_{\cA} < C \mu n^{h-1} 
\quad \text{ and } \quad
d_{\cB} < C \mu n^{2h-2}. 
\end{align}

Now we will bound $p_\cA$ and ${p_\cB}$ using switchings.
For $p_\cA$ we need to bound $\bP[A\mid \cap_{E \in {\cE'}} \ov{E}]$
for any $A=A(e,f,H_0) \in \cA$ and ${\cE'} \sub \mc{E}$
such that $AE \notin E(\GG)$ for all $E \in {\cE'}$
and $\bP[\cap_{E \in {\cE'}} \ov{E}]>0$.
Let $\cF$ be the set of $H$-factors of $G$
that satisfy $\cap_{E \in {\cE'}} \ov{E}$; then $\cF \ne \es$.
Let $\cF_0 = \{F_0 \in \cF: H_0 \in F_0\}$.
We consider the auxiliary bipartite multigraph $\cG_A$
with parts $(\cF_0, \cF \sm \cF_0)$, 
where for each $F_0 \in \cF_0$ and feasible 
$(H_0,F_0)$-switching $Y$ of size $m$
we add an edge from $F_0$ to $F$ 
obtained by replacing $F_0[V(Y)]$ with $Y$;
we note that $F \in \cF\sm \cF_0$ by Definition \ref{def:feas}.
Let $\dD_A$ be the minimum degree in  {$\cG_A$} of vertices in $\cF_0$
and $\DD_A$ be the maximum degree in {$\cG_A$} of vertices in $\cF \sm \cF_0$.
By double-counting the edges of $\cG_A$ we obtain
$\bP[A \mid \cap_{E \in  {\cE'}}  {\ov{E}} ] = |\cF_0|/|\cF| \le \DD_A/\dD_A$.

We therefore need an upper bound for $\DD_A$ and a lower bound for $\dD_A$.
By the hypotheses of the lemma, we have $\dD_A \ge \gG n^{m-1}$.
To bound $\DD_A$, we fix any $F\in \cF\sm\cF_0$ 
and bound the number of pairs $(F_0,Y)$ where $F_0 \in \cF_0$ 
and $Y$ is a feasible $(H_0,F_0)$-switching of size $m$ that produces $F$.
Each vertex of $V(H_0)$ must belong to a different copy of $H$ in $F$,
as otherwise there are no $(H_0,F_0)$-switchings that could produce $F$.
Thus we identify $h$ copies of $H$ in $F$ whose vertex set
must be included in $V(Y)$. There at most $n^{m-h}$ choices
for the other copies of $H$ to include in $V(Y)$
and then at most $(hm)!$ choices for $Y$, so $\DD_A \le (hm)! n^{m-h}$.
We deduce
\begin{align}\label{eq:pA}
\bP[A \vert \cap_{E \in  {\cE'}}  {\ov{E}} ] \leq (hm)! \gG^{-1} n^{1-h} =: p_\cA \;.
\end{align}

The argument to bound $p_\cB$ is very similar.
We need to bound $\bP[B\mid \cap_{E \in  {\cE'}} \ov{E}]$
for any $B=B(e_1,e_2,H_1,H_2) \in \cB$ and $ {\cE'} \sub \mc{E}$
such that $BE \notin E(\GG)$ for all $E \in  {\cE'}$
and $\bP[\cap_{E \in  {\cE'}} \ov{E}]>0$.
Let $\cF$ be the set of $H$-factors of $G$
that satisfy $\cap_{E \in  {\cE'}} \ov{E}$; then $\cF \ne \es$.
Let $\cF' = \{F' \in \cF: \{H_1,H_2\} \sub F'\}$.
We consider the auxiliary bipartite multigraph  {$\cG_B$}
with parts $(\cF', \cF \sm \cF')$, where there is an edge 
from $F' \in \cF'$ to $F$ for each pair $(Y,Z)$, 
where $Y$ is a feasible $(H_1,F')$-switching of size $m$
producing some $H$-factor $F''$
containing $H_2$ but not $H_1$,
and $Z$ is a feasible $(H_2,F'')$-switching of size $m$
with $V(Z) \cap V(H_1) = \es$ producing $F$;
note that then $F \in \cF \sm \cF'$.

We have $\bP[B \mid \cap_{E \in  {\cE'}}  {\ov{E}} ] \le \DD_B/\dD_B$,
where $\DD_B$ and $\dD_B$ are defined analogously to $\DD_A$ and $\dD_A$.
The condition $V(Z) \cap V(H_1) = \es$ 
rules out at most $hn^{m-2}$ choices of $Z$ given $H_1$,  {and similarly the condition
that $F''$ contains $H_2$ and not $H_1$ rules out at most $hn^{m-2}$ choices of $Y$ given $H_2$.}
So $\dD_B \ge (\gG n^{m-1} - hn^{m-2})^2 > \tfrac{1}{2}\gG^2 n^{2m-2}$.
Similarly to before we have $\DD_B \le ((hm)! n^{m-h})^2$, so
\begin{align}\label{eq:pB}
\bP[B \vert \cap_{E \in  {\cE'}}  {\ov{E}} ] \leq 2(hm)!^2 \gG^{-2} n^{2-2h} =: p_\cB \;.
\end{align}
Combining (\ref{eq:d}), (\ref{eq:pA}) and (\ref{eq:pB}) we have
$p_{\cA}d_{\cA}+p_{\cB}d_{\cB} \leq 1/4$, so the lemma follows from Lemma \ref{L4}.
\end{proof}

\section{Switchings} \label{sec:feas}

In this section we prove Lemma \ref{feas},
which shows how to obtain a feasible switching from 
a suitable partial $H$-factor and transverse partition
whose parts have high minimum degree.

\begin{proof}[Proof of Lemma \ref{feas}]
Let $F_0$ be an $H$-factor in $G$ and $H_0 \in F_0$.
Let $X \sub F_0$ be a suitable
partial $H$-factor in $G$ 
of size $m$ with $H_0 \in X$.
Let $\mc{P} = (V_1,\dots,V_h)$ be 
a transverse partition of $V(X)$ such that all
$\dD_\ell(G[V_i]) \geq (\dD^*_\ell(H) + \eps/2) m^{r-\ell}$.
We need to find a partial $H$-factor $Y$ in $G$ with $V(Y)=V(X)$
such that $Y$ is a feasible $(H_0,F_0)$-switching.

We construct $Y$ by successively choosing
$H$-factors $Y_i$ of $G[V_i]$ for $1 \le i \le h$.
For each $i$ we let $V'_i = V_i \sm V(H_0)$ and note that
$G[V'_i]$ is rainbow by Definition \ref{def:suit}.ii.
At step $i$, we let $G_i$ be the $r$-graph obtained from $G[V_i]$
by deleting all edges disjoint from $V(H_0)$ that share a colour
with any $H'$ in $F_0$ or $\cup_{j<i} Y_j$.
It suffices to show that $G_i$ has an $H$-factor $Y_i$,
as then $Y = \cup_{i=1}^h Y_i$ will be feasible.

By definition of $\dD^*_\ell(H)$, it suffices to show
for each $L \sub V_i$ with $|L|=\ell$ that
we delete at most $\tfrac{\eps}{2} m^{r-\ell}$ edges containing $L$.
We can assume $L$ is disjoint from $V(H_0)$,
as otherwise we do not delete any edges containing $L$.
There are $\tbinom{m-\ell}{r-1-\ell}$ choices of $I$
of size $r-1$ with $L \sub I \sub V_i$. 
For each such $I$, by Definition \ref{def:suit}.i,
the number of edges containing $I$ deleted due to sharing
a colour with any $H' \in F_0$ is at most $\eps m/4$.
Thus we delete at most $\tfrac{\eps}{4} m^{r-\ell}$
such edges containing $L$.

It remains to consider edges containing $L$
that are deleted due to sharing a colour
with any $H'$ in $\cup_{j<i} Y_j$.
As $G[V'_i]$ is rainbow, any colour in $\cup_{j<i} Y_j$
accounts for at most one deleted edge.
In the case $\ell \le r-2$ we can crudely bound
the number of deleted edges by the total number
of edges in $\cup_{j<i} Y_j$, which is at most 
$i e(H) m < m h^{r+1} < \tfrac{\eps}{4} m^{r-\ell}$.

Now we may suppose $\ell=r-1$. 
Consider any edge $e$ containing $L$ that is
deleted due to having the same colour as 
some edge $f$ in some $Y_j$ with $j<i$.
By Definition \ref{def:suit}.ii and $|e \sm L|=1$ there is
a copy $H'$ of $H$ in $X$ that intersects both $L$ and $f$.
To bound the number of choices for $e$,
note that there are $|L|=r-1$ choices for $H'$
and $i-1$ choices for $j$. These choices determine
a vertex in $V_j$, and so a copy of $H$ in $Y_j$,
which contains at most $h^{r-1}$ choices for $f$.
Then the colour of $f$ determines at most one deleted edge in $e$.
Thus the number of such deleted edges $e$ containing $L$ is at most
$(r-1)(i-1)h^{r-1} < \tfrac{\eps}{4} m$, as required.
\end{proof}

\section{Transverse partitions} \label{sec:shuffle}

To complete the proof of Theorem \ref{mainthm},
it remains to prove  Lemma \ref{shuffle},
which bounds the probability that a random partial $H$-factor
and transverse partition satisfy the hypotheses of Lemma \ref{feas}.

\begin{proof}[Proof of Lemma \ref{shuffle}]
Let $F_0$ be an $H$-factor in $G$ and $H_0 \in F_0$.
Let $X \sub F_0$ be a random partial $H$-factor 
where $H_0 \in X$ and each $H' \in F_0 \sm \{H_0\}$
is included independently with probability $p=\frac{m-1}{n/h-1} {\leq \frac{hm}{n}}$.
Let $\mc{P} = (V_1,\dots,V_h)$ be a uniformly random
transverse partition of $V(X)$. Note that each copy $H'$ of $H$ in $X$
has one vertex in each $V_i$, according to a uniformly random bijection
between $V(H')$ and $[h]$, and that these bijections are independent
for different choices of $H'$. Consider the events
\begin{align*}
\mc{E}_1 & = \{|X|=m\}, 
& \mc{E}_2 & = \{ X\ \text{satisfies Definition \ref{def:suit}.ii} \},\\
\mc{E}_3 & = \{ X\ \text{satisfies Definition \ref{def:suit}.i} \}, 
& \mc{E}_4 & = \cap_{i=1}^h \{ \dD_\ell(G[V_i]) 
\geq (\dD^*_\ell(H) + \eps/2) m^{r-\ell} \}.
\end{align*}
We need to show that $\mb{P}[\cap_{i=1}^4 \mc{E}_i] > 1/m$.
To do so, we first recall from Lemma~\ref{binbound}
that $\mb{P}[\mc{E}_1] \ge 1/(4\sqrt{m})$.
To complete the proof, we will show that
$\mb{P}[\mc{E}_i] \ge 1-1/m$ for $i=2,3,4$. Throughout, for $I \sub V(G)$ we let $F_I \sub F_0$ be the partial $H$-factor
consisting of all copies of $H$ in $F_0$ that intersect $I$.
\medskip

\nib{Bounding $\mb{P}[\mc{E}_2]$.}

For $s \in [r-1]$ let $\mc{Z}_s$ be the set of pairs $(e,f)$ 
of transverse edges disjoint from $V(H_0)$ of the same colour 
with $|e \cap f|=s$ and $X(e,f) = \es$.
As the colouring is $\mu$-bounded, we have 
$|\mc{Z}_s| \le n^r \cdot \tbinom{r}{s} \mu n^{r-s}$.
For any $(e,f) \in \mc{Z}_s$ we have $|F_{e \cup f}|=2r-s$,
so $\bP[e \cup f \sub V(X)] = p^{2r-s}$.
By a union bound, the probability that any such event holds is at most
$\sum_{s=1}^{r-1}  \tbinom{r}{s} \mu n^{2r-s} p^{2r-s}
< (hm)^r (hm+1)^r \mu < 1/2m$.

Similarly, let $\mc{Z}_0$ be the set of pairs $(e,f)$ 
of transverse edges disjoint from $V(H_0)$ of the same colour 
with $e \cap f=\es$ and $|X(e,f)| \le 1$.
As the colouring is $\mu$-bounded, we have 
$|\mc{Z}_0| \le n^r \cdot \mu n^{r-1}$.
For any $(e,f) \in \mc{Z}_0$, $|F_{e \cup f}| \ge 2r-1$ and $\bP[e \cup f \sub V(X)] \le p^{2r-1}$.
Thus the probability that any such event holds is at most
$\mu (hm)^{2r-1} < 1/2m$.

\medskip

\nib{Bounding $\mb{P}[\mc{E}_3]$.}

For any transverse $I \sub V(X) \sm V(H_0)$ with $|I|=r-1$
we let $B_I$ be the set of $v \in V(G) \sm (V(F_I) \cup V(H_0))$ such that 
$I \cup \{v\}$ is an edge sharing a colour with some $H' \in F_0$.
Write $Y_I = |V(X) \cap B_I|$. It suffices to bound the probability 
that there is any $I \sub V(X)$ with $Y_I > \eps m/5$. Indeed, the number of $v\in V(F_I) \cup V(H_0)$ such that 
$I \cup \{v\}$ is an edge is at most $rh < \eps m/20$.

First we show that $X$ is unlikely to contain any $I$
in $\mc{B} := \{ I: |B_I| > \eps n/10 {h} \}$. Indeed, as the colouring is $\mu$-bounded,
there are at most $e(F_0) \mu n^{r-1} = \mu e(H) n^r/h$  {edges with} colours in $F_0$,
so $|\mc{B}| < \mu \eps^{-2} n^{r-1}$.
For each transverse $I$ we have $\mb{P}[I \sub V(X)] = p^{r-1}$,
so by a union bound, the probability that $X$ contains any $I$ in $\mc{B}$ 
is at most $\mu \eps^{-2}  {(h m)}^{r-1} < 1/2m$.

Now for each $I \notin \mc{B}$ we bound $Y_I$ by Talagrand's inequality,
where the independent trials are the decisions for each
$H' \in F_0 \sm \{H_0\}$ of whether to include $H'$ in $X$.
As $I \notin \mc{B}$ we have $\bE[Y_I]= p |B_I| \leq \eps m/10$.
Also, $Y_I$ is clearly $h$-Lipschitz and $1$-certifiable.
We apply Lemma \ref{talagrand2} to $Y'_I = Y_I + \eps m/30$,
with $t = \eps m/30 \le \bE[Y'_I]$, $c=h$ and $r=1$ to deduce
$\bP[Y_I> \eps m/5] \leq 4e^{-10^{-4}h^{-2}\eps^2 m} < m^{-2r}$.

As we excluded $V(F_I)$ from $B_I$,
the events $\{I \sub V(X)\}$ and $Y_I > \eps m/5$ are independent,
so both occur with probability at most $p^{r-1} m^{-2r}$.
Taking a union bound over at most $n^{r-1}$ choices of $I$,
we obtain $\mb{P}[\ov{ {\cE}_3}] < 1/m$.

\medskip

\nib{Bounding $\mb{P}[\mc{E}_4]$.}

For $L \sub V(G)$ with $|L|=\ell$ and $i \in [h]$ we define
\[ \mc{J}_L  = \{J\sub V(G) \sm V(H_0): \ F_L\cap F_J=\es
\ \text{ and } \ L\cup J\in E(G) \text{ is transverse} \}.\]
We say $L$ is $i$-bad if $L \sub V_i$ and
$d_i'(L) := |\{J \in \mc{J}_L: J \sub V_i\}| 
< (\dD^*_\ell(H) + \eps/2) m^{r-\ell}$.
We will give an upper bound on the probability 
that there is any $i$-bad $L$.

First we note that the events
$\{L \sub V_i\}$ and $\{J \sub V_i\}$
are independent for any $J \in \mc{J}_L$.
There are at most $n^\ell$ choices of $L$
with $L \cap V(H_0) = \es$, each of which
has $\bP[L \sub V_i] = (p/h)^\ell$,
and at most $hn^{\ell-1}$ choices of $L$
with $|L \cap V(H_0)| = 1$, each of which has 
$\bP[L \sub V_i] \le (p/h)^{\ell-1}$.
By a union bound, it suffices to show 
for every transverse $L$ and $i \in [h]$ that
$\bP[d_i'(L) < (\dD^*_\ell(H) + \eps/2)  {m}^{r-\ell}] < m^{-2r}$.

We also note that $|\mc{J}_L| \ge (\dD^*_\ell(H) + 0.9\eps) n^{r-\ell}$, 
as there are at least $(\dD^*_\ell(H) + \eps) n^{r-\ell}$
choices of $J$ with $L \cup J \in E(G)$, 
of which the number excluded
due to $J \cap V(H_0) \ne \es$, $F_L\cap F_J \ne \es$
or $L \cup J$ not being transverse is at most
$hn^{r-\ell-1} + \ell hn^{r-\ell-1} 
+ \tfrac{n}{h} \tbinom{h}{2} n^{r-\ell-2} < 0.1 \eps n^{r-\ell}$.

We will apply Janson's inequality to $d_i'(L) = \sum_{J \in \mc{J}_L} I_J$, 
where each $I_J$ is the indicator of $\{J \sub V_i\}$.
As $\mb{P}[J \sub V_i] = (p/h)^{r-\ell}$ for each $J \in \mc{J}_L$,
we have $\mu = \mb{E}[d_i'(L)] > (\dD^*_\ell(H) + 0.9\eps) m^{r-\ell}$.
We use the dependency graph $\GG$ where 
$JJ'$ is an edge iff $F_J \cap F_{J'} \ne \es$.
Note that for any $J \in \mc{J}_L$ and $s \in [r-\ell]$
the number of choices of $J'$ with $|F_J \cap F_{J'}|=s$
is at most $\tbinom{r-\ell}{s} h^s n^{r-\ell-s}$, and for each
we have $\mb{P}[J \cup J' \sub V_i] =  {(p/h)}^{2(r-\ell)-s}$.
Thus we can bound the parameter $\DD$ in Theorem \ref{janson} as
$\DD \le |\mc{J}_L|  {\sum_{s=1}^{r-\ell}}\tbinom{r-\ell}{s} h^s n^{r-\ell-s}   {(p/h)}^{2(r-\ell)-s}
\le m^{r-\ell}\sum_{s=1}^{r-\ell}{{r-\ell \choose s}h^sm^{r-\ell-s} <  2h (r-\ell) m^{2(r-\ell)-1}}$.
We also have
$\dD \le  {\sum_{s=1}^{r-\ell}}\tbinom{r-\ell}{s} h^s n^{r-\ell-s}  {(p/h)}^{r-\ell-s} 
\le  {\sum_{s=1}^{r-\ell}{r-\ell \choose s}h^sm^{r-\ell-s}} <  {2h(r-l)m^{r-\ell-1}}$.
By Theorem~\ref{janson}, there is some constant 
$c=c(r,\eps,h)$ independent of $m$ so that
$\bP[d_i'(L) < (\dD^*_\ell(H) + \eps/2) {m}^{r-\ell}] 
< e^{-cm} < m^{-2r}$, as required.
\end{proof}

\section{Concluding remarks}

Our result and those of \cite{CG,GJ} suggest that
for any Dirac-type problem, the rainbow problem 
for bounded colourings should have asymptotically
the same degree threshold as the problem with no colours.
In particular, it may be interesting to establish this
for Hamilton cycles in hypergraphs
(i.e.\ a Dirac-type generalisation of \cite{DF}). 
The local resilience perspective 
emphasises analogies with the recent literature
on Dirac-type problems in the random setting (see the surveys \cite{B, sudakov}), 
perhaps suggests looking for common generalisations, 
e.g.\ a rainbow version of \cite{leesudakov}: 
in the random graph $G(n,p)$ with $p > C(\log n)/n$,
must any $o(pn)$-bounded edge-colouring of any subgraph $H$ 
with minimum degree $(1/2+o(1))pn$ have a rainbow Hamilton cycle?


\begin{thebibliography}{99}

\bibitem{B} J. B\"ottcher,
Large-scale structures in random graphs,
{\em Surveys in Combinatorics},
Cambridge University Press, 87--140, 2017.

\bibitem{BST} J. B\"ottcher, M. Schacht and A. Taraz, 
Proof of the bandwidth conjecture 
of Bollob\'as and Koml\'os, 
{\em Math. Ann.} 343:175--205, 2009.

\bibitem{brualdi} R.A. Brualdi and H.J. Ryser, 
Combinatorial matrix theory, 
Cambridge University Press, 1991.

\bibitem{CG}  M. Coulson and G. Perarnau,
Rainbow matchings in Dirac bipartite graphs, arXiv:1711.02916.

\bibitem{DF} A. Dudek and M. Ferrara, 
Extensions of Results on 
Rainbow Hamilton Cycles in Uniform Hypergraphs,
{\em Graphs Combin.} 31:577--583, 2015.

\bibitem{ematch} P. Erd\H{o}s, 
A problem on independent $r$-tuples, 
{\em Ann. Univ. Sci. Budapest}, 8:93--95, 1965.

\bibitem{ES} P. Erd\H{o}s and J. Spencer, 
Lopsided Lov{\'a}sz local lemma and Latin transversals,
{\em Disc. Applied Math.} 30:151--154, 1991.

\bibitem{GJ} S. Glock and F. Joos,
A rainbow blow-up lemma, arXiv:1802.07700.

\bibitem{GS} R. Graham and N. Sloane, 
On additive bases and harmonious graphs,
{\em SIAM J. Alg. Disc. Methods} 1:382--404, 1980.

\bibitem{J} S. Janson,
New versions of Suen's correlation inequality,
{\em Random Struct. Alg.} 13:467--483, 1998.

\bibitem{Kturan}  
P. Keevash, Hypergraph Tur\'an Problems, {\em Surveys in Combinatorics},
Cambridge University Press, 83--140, 2011. 

\bibitem{KMSV}
P. Keevash, D. Mubayi, B. Sudakov and J. Verstra\"ete,
Rainbow Tur\'an Problems,
{\em Combin. Probab. Comput.} 16:109--126, 2007.

\bibitem{KSS} J. Koml\'os, G.N. S\'ark\"ozy and E. Szemer\'edi,
Blow-up lemma, {\em Combinatorica} 17:109--123, 1997.

\bibitem{KSSalonyus} J. Koml\'os, G.N. S\'ark\"ozy and E. Szemer\'edi,
Proof of the Alon-Yuster conjecture, {\em Disc. Math.} 235:255--269, 2001.

\bibitem{KOicm} D. K\"uhn and D. Osthus,
Hamilton cycles in graphs and hypergraphs: an extremal perspective, 
{\em Proc. ICM 2014} 4:381--406, Seoul, Korea, 2014.

\bibitem{KOfactors} D. K\"uhn and D. Osthus,
The minimum degree threshold for perfect graph packings, 
{\em Combinatorica} 29:65--107, 2009.

\bibitem{KOlarge} D. K\"uhn and D. Osthus,
Embedding large subgraphs into dense graphs,
{\em Surveys in Combinatorics}, 
Cambridge University Press, 137--167, 2009.

\bibitem{leesudakov}
C. Lee and B. Sudakov, Dirac's theorem for random graphs, 
{\em Random Struct. Alg.}, 41:293--305, 2012.

\bibitem{MR} M. Molloy and B. Reed, 
{\em Graph colouring and the probabilistic method},
Springer Science \& Business Media, 2013.

\bibitem{MPS} R. Montgomery, A. Pokrovskiy and B. Sudakov,
Embedding rainbow trees with applications to graph 
labelling and decomposition, arXiv:1803.03316.

\bibitem{R} G. Ringel, 
Theory of graphs and its applications,
in {\em Proc. Symposium Smolenice}, 1963.

\bibitem{RR} V.~R\"odl and A.~Ruci\'nski, Dirac-type questions for hypergraphs 
--- a survey (or more problems for Endre to solve),
\emph{An Irregular Mind (Szemer\'edi is 70)} 21:1--30, 2010.

\bibitem{ryser} H. Ryser, Neuere Probleme in der Kombinatorik,
Vortrage \"uber Kombinatorik, Oberwolfach, 69--91, 1967. 

\bibitem{S} J. Spencer, 
Asymptotic lower bounds for Ramsey functions,
{\em Disc. Math.} 20:69--76, 1977.

\bibitem{stein} S. K. Stein, 
Transversals of Latin squares and their generalizations,
{\em Pacific J. Math.} 59:567--575, 1975.

\bibitem{sudakov}
B. Sudakov, Robustness of graph properties, 
{\em Surveys in Combinatorics},
Cambridge University Press, 372--408, 2017.

\bibitem{SV} B. Sudakov and V.H. Vu, 
Local resilience of graphs, 
{\em Random Struct. Alg.} 33:409--433, 2008.

\bibitem{yi} Y. Zhao, 
Recent advances on Dirac-type problems for hypergraphs, 
in: {\em Recent Trends in Combinatorics}, Springer, 2016.

\bibitem{yuster} R. Yuster, Rainbow H-factors, 
{\em Electron. J. Combin.} 13:R13, 2006.

\end{thebibliography}
\end{document}